\documentclass[12pt]{amsart}
\usepackage{fullpage}
\usepackage{amsfonts,amscd}
\usepackage{amssymb}
\usepackage{url}
\usepackage{graphicx}
\usepackage[english]{babel}
\theoremstyle{plain}
\newtheorem{theorem}                {Theorem}      [section]
\newtheorem{proposition}  [theorem]  {Proposition}

\theoremstyle{definition}
\newtheorem{example}      [theorem]  {Example}

\newtheorem{definition}   [theorem]  {Definition}

\numberwithin{equation}{section}

\def \R{{\mathbb R}}
\def \s{{\mathbb S}}

\usepackage{color}

\DeclareMathOperator{\cst}{constant}
\DeclareMathOperator{\grad}{grad}
\numberwithin{equation}{section}

\begin{document}

\title[]{Biharmonic curves into quadrics}

\author{S.~Montaldo}
\address{Universit\`a degli Studi di Cagliari\\
Dipartimento di Matematica e Informatica\\
Via Ospedale 72\\
09124 Cagliari, Italia}
\email{montaldo@unica.it}

\author{A.~Ratto}
\address{Universit\`a degli Studi di Cagliari\\
Dipartimento di Matematica e Informatica\\
Viale Merello 93\\
09123 Cagliari, Italia}
\email{rattoa@unica.it}

\begin{abstract}
We develop an essentially algebraic method to study biharmonic curves into an implicit surface. Although our method is rather general, it is especially suitable to study curves into surfaces defined by a polynomial equation: in particular, we
use it to give a complete classification of biharmonic curves into real quadrics of the 3-dimensional Euclidean space.
\end{abstract}



\thanks{Work supported by P.R.I.N. 2010/11 -- Variet\`a reali e complesse: geometria, topologia e analisi armonica -- Italy.}

\maketitle

\section{Introduction}\label{intro}

Biharmonic curves $\gamma:I\subset\R\to (N,h)$ of a
Riemannian manifold are the solutions of the fourth
order differential equation
\begin{equation}\label{eq-main}
\nabla^3_{\gamma'}\gamma' - R(\gamma',\nabla_{\gamma'} \gamma')\gamma' = 0,
\end{equation}
where $\nabla$ is the Levi-Civita connection on $(N,h)$ and $R$ is its curvature operator.\\

\noindent As we shall detail in the next  section,
they arise from a variational problem and are a
natural generalization of geodesics. In the last decade biharmonic curves have been extensively studied and classified in several spaces by analytical inspection of \eqref{eq-main} (see, for example, \cite{RCSMCO1,RCSMPP1,COP,CMOP,JIJL,JCJIJL,Dim,DF1,DF2,SMIO}) .\\

Although much work has been done, the full understanding
of biharmonic curves in a surface of the Euclidean three-dimensional space is far from been achieved. As yet, we have a clear picture of biharmonic curves in a surface only in the case that the surface is invariant by the action of a one parameter group of isometries of the ambient space. For example, in \cite{RCSMPP1} it was
proved that a biharmonic curve on a surface of revolution in the Euclidean space (invariant by the action
of $SO(2)$) must be a parallel, that is an orbit of the action of the group on the surface. This property was then generalized to invariant surfaces in  a $3$-dimensional manifold \cite{SMIO}.\\

The main obstacle in trying to describe and classify biharmonic curves on a surface by analytical methods is that  \eqref{eq-main} is a fourth order differential equation which is very hard to tackle. \\

In this paper we propose a scheme to classify biharmonic curves on a quadric in the three-dimensional Euclidean space by using only algebraic methods. The main point is that a quadric can be described implicitly by a polynomial equation $F(x,y,z)=0$ and we will show that the biharmonic candidates must be the intersection of the given quadric with another specific algebraic surface $G(x,y,z)=0$. The latter property allows us to classify the biharmonic curves into any non degenerate quadric. \\

In the last paragraph we give some examples to show how to use this approach for other implicit surfaces.

\section{Preliminaries}

{\it Harmonic maps}  are critical points of the {\em energy} functional
\begin{equation}\label{energia}
E(\varphi)=\frac{1}{2}\int_{M}\,|d\varphi|^2\,dv_g \,\, ,
\end{equation}
where $\varphi:(M,g)\to(N,h)$ is a smooth map between two Riemannian
manifolds $M$ and $N$. In analytical terms, the condition of harmonicity is equivalent to the fact that the map $\varphi$ is a solution of the Euler-Lagrange equation associated to the energy functional \eqref{energia}, i.e.
\begin{equation}\label{harmonicityequation}
    {\rm trace} \, \nabla d \varphi =0 \,\, .
\end{equation}
The left member of \eqref{harmonicityequation} is a vector field along the map $\varphi$, or, equivalently, a section of the pull-back bundle $\varphi^{-1} \, (TN)$: it is called {\em tension field} and denoted $\tau (\varphi)$.

A related topic of growing interest deals with the study of the so-called {\it biharmonic maps}: these maps, which provide a natural generalisation of harmonic maps, are the critical points of the bienergy functional (as suggested by Eells--Lemaire \cite{EL83})
\begin{equation}\label{bienergia}
    E_2(\varphi)=\frac{1}{2}\int_{M}\,|\tau (\varphi)|^2\,dv_g \,\, .
\end{equation}
In \cite{Jiang} Jiang derived the first variation and the second variation formulas for the bienergy. In particular, he showed that the Euler-Lagrange equation associated to $E_2(\varphi)$ is
\begin{equation}\label{bitensionfield}
    \tau_2(\varphi) = - J\left (\tau(\varphi) \right ) = - \triangle \tau(\varphi)- \rm{trace} R^N(d \varphi, \tau(\varphi)) d \varphi = 0 \,\, ,
    \end{equation}
where $J$ denotes (formally) the Jacobi operator of $\varphi$, $\triangle$ is the rough Laplacian on sections of $\varphi^{-1} \, (TN)$ that, for a local orthonormal frame $\{e_i\}_{i=1}^m$ on $M$, is defined by
$$
\Delta=-\sum_{i=1}^m\{\nabla^{\varphi}_{e_i}\nabla^{\varphi}_{e_i}-\nabla^{\varphi}_{\nabla^{M}_{e_i}e_i}\}\,,
$$
and
\begin{equation}\label{curvatura}
    R^N (X,Y)= \nabla_X \nabla_Y - \nabla_Y \nabla_X -\nabla_{[X,Y]}
\end{equation}
is the curvature operator on $(N,h)$.

We point out that \eqref{bitensionfield} is a {\it fourth order}  semi-linear elliptic system of differential equations. We also note that any harmonic map is an absolute minimum of the bienergy, and so it is trivially biharmonic. Therefore, a general working plan is to study the existence of biharmonic maps which are not harmonic:  these shall be referred to as {\it proper biharmonic maps}. We refer to \cite{SMCO} for existence results and general properties of biharmonic maps. \\

	Let now
$\gamma:I\to (N,h)$ be a curve parametrized by arc length, from an open
interval $I\subset\R$ to a Riemannian
manifold. In this case, putting $T=\gamma'$, the tension field
becomes
$ \tau({\gamma})=\nabla_{T}T$ and the
biharmonic equation \eqref{bitensionfield} reduces to
\begin{equation}\label{bieq-curve}
\nabla^3_{T}T - R(T,\nabla_T T)T = 0\,\, .
\end{equation}

\noindent In order to describe geometrically equation \eqref{bieq-curve},
let us recall the definition of the Frenet frame.
\begin{definition}[See, for example, \cite{Laugwitz}] \label{def2.1}
The Frenet frame $\{F_{i}\}_{i=1,\dots,n}$ associated
to a curve
$\gamma : I\subset {\R}\to (N^{n},h)$,  parametrized by arc
length, is the orthonormalisation of
the $(n+1)$-uple  $\{ \nabla_{\frac{\partial}{\partial
t}}^{(k)} d\gamma
(\frac{\partial}{\partial t})
\}_{k=0,\dots,n}$ described by:
\begin{align*} F_{1}&=d\gamma
(\frac{\partial}{\partial t}) ,  \\
\nabla_{\frac{\partial}{\partial
t}}^{\gamma} F_{1} &= k_{1} F_{2} , \\
\nabla_{\frac{\partial}{\partial
t}}^{\gamma} F_{i} &= - k_{i-1} F_{i-1}
+ k_{i}F_{i+1} , \quad \forall i =
2,\dots,n-1 , \\
\nabla_{\frac{\partial}{\partial
t}}^{\gamma} F_{n} &= - k_{n-1} F_{n-1}
, \end{align*}
where the functions $\{k_{1},k_{2},\ldots,k_{n-1}\}$
are called the {\em curvatures} of $\gamma$ and $\nabla^{\gamma}$
is the Levi-Civita connection on the pull-back bundle $\gamma^{-1}(TN)$. Note that
$F_{1}=T={\gamma}'$ is the unit tangent vector field along the curve.
 \end{definition}

\noindent Using the Frenet frame, the biharmonic equation
\eqref{bieq-curve} reduces to a differential system invol\-ving the
curvatures of $\gamma$ and if we look for proper biharmonic solutions,
that is for biharmonic  curves with $k_{1}\neq0$, we have

\begin{proposition}[\cite{CMOP}]\label{prop1}
Let $\gamma : I \subset {\R} \to (N^{n},h)$ ($n\geq 2$) be
a curve parametrized by arc length from an open interval
of $\R$ into an n-dimensional Riemannian manifold $(N^n,h)$.
Then $\gamma$ is proper biharmonic if and only if:
\begin{equation}\label{euler-lagrange}
\left\{
\begin{array}{l}
k_{1} = \cst\neq 0\\
k_{1}^2 + k_2^2 =  R(F_1,F_2,F_1,F_2) \\
{k'_2} = -R(F_1,F_2,F_1,F_3) \\
k_2 k_3 = - R(F_1,F_2,F_1,F_4) \\
R(F_1,F_2,F_1,F_j) = 0\,\, , \hspace{1,5 cm} j = 5,\ldots, n\,\,.
\end{array}
\right.
\end{equation}
\end{proposition}

As a special case of \eqref{euler-lagrange}, if $\gamma : I \subset {\R} \to (N^{2},h)$ is a curve into a surface, then $\gamma$ is proper biharmonic if and only if

\begin{equation}\label{euler-lagrange-curve-surface}
\left\{
\begin{array}{l}
k_{1} = \cst\neq 0\\
k_{1}^2=K\,\,,
\end{array}
\right.
\end{equation}
where $K$ is the Gaussian curvature of the surface $(N^2,h)$.

\section{Formulas for the curvatures of implicit surfaces and implicit curves}

Let $F:\R^3\to\R$ be a differentiable function: we shall assume  that, for all $p\in N^2=F^{-1}(0)$, $(\grad F)({p})\neq 0$, so that $N^2$ is a regular surface in $\R^3$. If we denote by
$C_{HF}$ the cofactor matrix of the Hessian $HF$ of $F$,
the Gaussian curvature of the surface $N^2$ is given by (see, for example, \cite{RG})
\begin{equation}\label{eq:gauss-curvature}
K=\frac{(\grad F) (C_{HF}) (\grad F)^{\top}}{\|\grad F\|^4}\,\,.
\end{equation}

Let now $F:\R^3\to\R$ and $G:\R^3\to\R$ be two differentiable functions such that $F^{-1}(0)$ and $G^{-1}(0)$ are, as above, two regular surfaces in $\R^3$, and also assume that in all points $p\in F^{-1}(0)\cap G^{-1}(0)$ the gradients
$\grad F$ and $\grad G$ are linearly independent. Then $F^{-1}(0)\cap G^{-1}(0)$
defines the trace of a regular curve in $\R^3$ that, locally, can be parametrized
by arc length as $\gamma(s)=(x(s),y(s),z(s))$, $s\in(a,b)$. The unit tangent vector to $\gamma$ is then
$$
\gamma'(s)=\frac{d\gamma}{ds}=T=\frac{\grad F\wedge\grad G}{\|\grad F\| \|\grad G\|}\,\, .
$$

The curve $\gamma$ can be seen as a curve both of $F^{-1}(0)$ and $G^{-1}(0)$.
For each point $p=\gamma(s)$, $s\in(a,b)$, we denote by $k_n^F(p)$ (respectively $k_n^G(p)$) the normal curvature at $p$ of the surface $F^{-1}(0)$ (respectively $G^{-1}(0)$) in the direction of $T$.

The curvature $k(s)$ of the curve $\gamma:(a,b)\to\R^3$ can be computed in terms of the normal curvatures  $k_n^F(p)$ and $k_n^G(p)$, $p=\gamma(s)$, as
\begin{equation}\label{eq:curvature-gamma}
k^2=\frac{1}{\sin^2\vartheta} \left( (k_n^F)^2+(k_n^G)^2-2(k_n^F) (k_n^G)\cos\vartheta\right)\,,
\end{equation}
where $\vartheta$ is the angle between $(\grad F)(p)$ and $(\grad G)(p)$, that is
$$
\cos\vartheta=\frac{\langle \grad F, \grad G\rangle}{\|\grad F\| \|\grad G\|}\,\,.
$$
The proof of \eqref{eq:curvature-gamma} is immediate. In fact, $k(s)$ is the norm
of $\gamma''(s)=d^2\gamma/ds^2$ which is normal to $T$. Thus

$$
\gamma''=\alpha \frac{\grad F}{\|\grad F\|}+ \beta \frac{\grad G}{\|\grad G\|}
$$ for some functions $\alpha,\beta:(a,b)\to \R$ which, recalling that
$$k_n^F=\langle \gamma'', \frac{\grad F}{\|\grad F\|}\rangle\,,\quad k_n^G=\langle \gamma'', \frac{\grad G}{\|\grad G\|}\rangle\,\,,$$
can be expressed by:
$$
\alpha=\frac{k_n^F-k_n^G \cos\vartheta}{\sin^2\vartheta}\,,\quad \beta=\frac{k_n^G-k_n^F \cos\vartheta}{\sin^2\vartheta}\,\,.
$$

Finally, looking at $\gamma(s)$ as a curve in the surface $F^{-1}(0)$, at a point $p=\gamma(s)$
the geodesic curvature $k_1(s)$, the normal curvature $k_n^F(p)$ and the curvature $k(s)$ are related by the formula
\begin{equation}\label{eq:curvature-normal-geodesic}
k^2=k_1^2+(k_n^F)^2\,\,.
\end{equation}

Thus, combining  \eqref{eq:curvature-gamma} and \eqref{eq:curvature-normal-geodesic}, we have the following proposition.

\begin{proposition}\label{pro:geodesic}
Let $F:\R^3\to\R$ and $G:\R^3\to\R$ be two differentiable functions such that $F^{-1}(0)$ and $G^{-1}(0)$ are two regular surfaces in $\R^3$. Assume that in all points $p\in F^{-1}(0)\cap G^{-1}(0)$ the gradients
$\grad F$ and $\grad G$ are linearly independent. Then the geodesic curvature $k_1$ of the curve $\gamma:(a,b)\to F^{-1}(0)\subset\R^3$, with $\gamma(s)\in F^{-1}(0)\cap G^{-1}(0)$, for all $s\in(a,b)$,
 is given by
\begin{equation}\label{eq:geodesic}
k_1^2=\frac{(\cos\vartheta k_n^F - k_n^G)^2}{\sin^2\vartheta}\,\,.
\end{equation}
\end{proposition}

Moreover, taking twice the derivative of  $F(\gamma(s))=0$, we find
$$
0=\frac{d^2 F}{ds^2}=T (HF) T^{\top}+ \langle \grad F, \gamma''\rangle\,\,.
$$
It follows that
\begin{equation}\label{eq:knf}
k_n^F=-\frac{T (HF) T^{\top}}{\|\grad F\|}\,
\end{equation}
and, similarly,
\begin{equation}\label{eq:kng}
k_n^G=-\frac{T (HG) T^{\top}}{\|\grad G\|}\,\,.
\end{equation}

The main point here is that, to compute the geodesic curvature of the curve $\gamma$ defined as in Proposition~\ref{pro:geodesic}, there is no need to parametrize the intersection curve, because  \eqref{eq:geodesic} can be explicitly
written in terms of $\grad F$, $\grad G$ and the Hessian matrices $HF$ and $HG$.

\section{Biharmonic curves into real quadrics}

Let $\mathcal Q$ be a real, non degenerate quadric in $\R^3$. Then, with respect to an adapted orthonormal frame of $\R^3$,  $\mathcal Q=F^{-1}(0)$ where
\begin{equation}\label{eq:quadric-center}
F(x,y,z)= \frac{x^2}{a^2}+\xi \frac{y^2}{b^2}+\zeta \frac{z^2}{c^2}-1\,\,,\quad \xi,\zeta =\pm 1\; \text{and}\; a,b,c>0
\end{equation}
if $\mathcal Q$ is a quadric with center, or
\begin{equation}\label{eq:quadric-no-center}
F(x,y,z)= \frac{x^2}{a^2}+\eta \frac{y^2}{b^2}-2z\,,\quad \eta=\pm 1\; \text{and}\; a,b>0
\end{equation}
otherwise.

According to \eqref{euler-lagrange-curve-surface}, the Gauss curvature of the surface along a proper biharmonic curve must be a positive constant. If we compute the Gauss curvature of a quadric using \eqref{eq:gauss-curvature}
we get
\begin{equation}\label{eq:gauss-curvature-quadric-center}
K=\frac{\xi\,\zeta }{a^2\, b^2\, c^2 \left(\dfrac{ x^2}{a^4}+\dfrac{ y^2 }{b^4}+\dfrac{ z^2}{c^4}\right)^2}
\end{equation}
for the quadrics with center and
\begin{equation}\label{eq:gauss-curvature-quadric-no-center}
K=\frac{\eta }{a^2 b^2 \left(\dfrac{x^2}{a^4}+\dfrac{y^2 }{b^4}+1\right)^2}
\end{equation}
otherwise.  Thus a quadric with center can admit a proper biharmonic curve only if $\xi\,\zeta>0$. If $\xi=\zeta=1$ and $a=b=c$, then the quadric is a sphere and the proper biharmonic curves are the circles of radius $\sqrt{2} a/2$, a result proved in \cite{RCSMPP1}. In all the other cases, combining \eqref{euler-lagrange-curve-surface} and \eqref{eq:gauss-curvature-quadric-center}, we conclude that if there exists a proper biharmonic curve, it must be the intersection of the given quadric with an ellipsoid of the type
\begin{equation}\label{eq:ellipsoid-d}
\dfrac{ x^2}{a^4}+\dfrac{ y^2 }{b^4}+\dfrac{ z^2}{c^4}=d^2\,\, ,
\end{equation}
where $d\in\R$ . Similarly, a quadric without center can admit a  proper biharmonic curve only if $\eta>0$ . In this case, the biharmonic curve, if there exists, must be the intersection of the given quadric with a cylinder of the type
\begin{equation}\label{eq:cylinder-e}
\dfrac{x^2}{a^4}+\dfrac{y^2 }{b^4}=e^2-1\,\, ,
\end{equation}
where $e\in\R$ . \\

We are now in the right position to state the main result of the paper.

\begin{theorem} Let $\mathcal Q$ be a non degenerate quadric which is not a sphere.
\begin{itemize}
\item[(a)] If $\mathcal Q$ is a quadric with center (as in \eqref{eq:quadric-center}), then $\mathcal Q$ admits a proper biharmonic curve if and only if
\begin{equation}\label{condizioneesistenza}
    \xi=\zeta=1 \quad  {\rm and} \quad a=b\,\,.
\end{equation}
Moreover, if \eqref{condizioneesistenza} holds, the biharmonic curve is the intersection of the quadric $\mathcal Q$ with the ellipsoid \eqref{eq:ellipsoid-d} with $d^2=1/(ac)$.
\item[(b)]   If $\mathcal Q$ is a quadric without center (as in \eqref{eq:quadric-no-center}), then $\mathcal Q$ does not admit any proper biharmonic curve.
\end{itemize}
\end{theorem}
\begin{proof}
 We shall begin considering  quadrics with center.
As we proved above, if there exists a proper biharmonic curve $\gamma$, it must be the intersection of $\mathcal Q$ with an ellipsoid \eqref{eq:ellipsoid-d}, i.e.
\begin{equation}\label{eq:quadric-intersection}
\gamma:\;\begin{cases}
F(x,y,z)= \dfrac{x^2}{a^2}+\xi \dfrac{y^2}{b^2}+\zeta \dfrac{z^2}{c^2}-1=0\vspace{3mm}\\
G(x,y,z)=\dfrac{ x^2}{a^4}+\dfrac{ y^2 }{b^4}+\dfrac{ z^2}{c^4}-d^2=0\,\,,
\end{cases}
\end{equation}
with $\xi\,\zeta>0$. Suppose that $\xi=\zeta=1$: then, using \eqref{eq:geodesic}, we can compute the geodesic curvature of the intersection curve $\gamma$ as a curve of the quadric $\mathcal Q$. A long, but straightforward, computation yields:
\begin{equation}\label{eq:k1-ellipsoid}
  k_1^2=\frac{1}{d^2} \frac{\left[d^2 \lambda_4-\left( \frac{x^2}{a^6}+\frac{y^2}{b^6}+\frac{z^2}{c^6}\right)\lambda_6\right]^2}{\lambda_8^2\left[d^2 \left(\frac{x^2}{a^8}+\frac{y^2}{b^8}+\frac{z^2}{c^8}\right) - \left( \frac{x^2}{a^6}+\frac{y^2}{b^6}+\frac{z^2}{c^6}\right)^2\right]}\,\,,
 \end{equation}
 where
 $$
 \lambda_n=a^n y^2 z^2 \left(b^2-c^2\right)^2+b^n x^2 z^2 \left(a^2-c^2\right)^2+c^n x^2 y^2 \left(a^2-b^2\right)^2\,\,.
 $$

 Now,  since $\mathcal Q$ is not a sphere, we can assume that $a\geq b>c$. Moreover, the curve $\gamma$ is a real curve with infinity points if and only if $d^2 c^2 -1>0$ . Under these hypotheses the curve $\gamma$ can be parametrized by
\begin{equation}\label{eq:gamma-u-ellipsoid}
 \gamma(u)=\begin{cases}
 x(u)= r_1 \cos u\\
 y(u)= r_2 \sin u\\
 z(u)= c \sqrt{1 - (r_1^2 \cos^2u)/a^2 - (r_2^2 \sin^2u)/b^2}\,\,,
 \end{cases}
\end{equation}
where
$$
r_1=a^2 \sqrt{\frac{1-c^2 d^2}{a^2-c^2}}\,\, ,\,\quad r_2=b^2 \sqrt{\frac{1-c^2 d^2}{b^2-c^2}}\,\,.
$$
Now, replacing \eqref{eq:gamma-u-ellipsoid} in \eqref{eq:k1-ellipsoid}, we obtain
\begin{equation}\label{eq:k1-u-ellipsoid}
k_1^2=\frac{8\left[A+4(c^2 d^2-1)B \cos 2u +(a^2-b^2)(c^2 d^2-1) \cos 4u \right]^2}{d^2 (c^2 d^2-1)\left[C+4 D \cos 2u - (a^2-b^2)(c^2 d^2 -1) \cos 4u\right]^3}\,\,,
\end{equation}
where $A,B,C,D$ are real constants given by
$$
A= -8 a^4 b^4 d^4+8 a^4 b^2 d^2+3 a^4 c^2 d^2-3 a^4+8 a^2 b^4 d^2-6 a^2 b^2 c^2 d^2-2 a^2 b^2+3 b^4 c^2 d^2-3 b^4\,\,,
$$
$$
B=a^4 \left(2 b^2 d^2-1\right)-2 a^2 b^4 d^2+b^4\,\,,
$$
$$
C=-4 a^4 b^2 d^2+a^4 c^2 d^2+3 a^4-4 a^2 b^4 d^2-2 a^2 b^2 c^2 d^2+2 a^2 b^2+b^4 c^2 d^2+3 b^4\,\,,
$$
$$
D=a^4 \left(b^2 d^2-1\right)-a^2 b^4 d^2+b^4\,\,.
$$
Next, by setting $w=\sin^2 u$ in \eqref{eq:k1-u-ellipsoid}, it is easy to conclude that, in terms of this new variable,
\begin{equation}\label{eq:k1-u-ellipsoidbis}
    k_1^2=\frac{N(w)}{D(w)} \,\, ,
\end{equation}
where the numerator and the denominator of \eqref{eq:k1-u-ellipsoidbis} are polynomials of degree $8$ and $12$ respectively. At this stage, direct inspection of the leading terms shows that $k_1$ can be a constant only if $a=b$. In this case, the condition that the curve $\gamma$ is proper biharmonic, that is $k_1^2-K=0$, becomes
$$
\frac{1-a^2 c^2 d^4}{a^4 c^2 d^4 \left(c^2 d^2-1\right)}=0
$$
from which the desired result follows.\\

When $\xi=\zeta=-1$ the computations are similar to the previous case. First, we point out that in this case
 $\gamma$ is a real curve with infinity points if and only if  $a^2 d^2-1>0$. Moreover, by means of a computation similar to \eqref{eq:k1-u-ellipsoid}, we can conclude that
if the geodesic curvature $k_1$ of $\gamma$ is constant then $b=c$. Next,
when $b=c$, the biharmonic condition $k_1^2-K=0$ reads now as
$$
\frac{a^2 c^2 d^4+1}{a^2 c^4 d^4 \left(a^2 d^2-1\right)}=0\,\, ,
$$
which has no real solution, so ending the case of quadrics with center. \\

When $\mathcal Q$ is a quadric without center, as in \eqref{eq:quadric-no-center}, a proper biharmonic curve $\gamma$, if it exists, must be, as we have remarked above, the intersection of $\mathcal Q$ with a cylinder, i.e.
\begin{equation}\label{eq:quadric-intersection}
\gamma:\;\begin{cases}
F(x,y,z)= \dfrac{x^2}{a^2}+\dfrac{y^2}{b^2}-2z=0\,,\vspace{3mm}\\
G(x,y,z)=\dfrac{x^2}{a^4}+\dfrac{y^2 }{b^4}-e^2+1=0\,\,,
\end{cases}
\end{equation}
where $e^2 > 1$ . Now, our method leads us to the following expression for the geodesic curvature $k_1$ of $\gamma$ :
\begin{equation}\label{eq:k1-u-ellipsoidtris}
  k_1^2=\frac{(\lambda_6^2-a^6b^6e^2\lambda_4)^2}{e^2\left[\lambda_8+x^2y^2(a^2-b^2)^2\right]^2(a^4b^4e^2\lambda_8-\lambda_6^2)}\,,
\end{equation}
where
$$
\lambda_n=b^nx^2+a^ny^2\,.
$$
In this case we propose a purely algebraic inspection of \eqref{eq:k1-u-ellipsoidtris} to show that $k_1^2$ is constant if and only if $a=b$.
First we observe that the points
$$
P_1=\left(0, b^2 \sqrt{e^2-1}, \frac{b^2}{2} (e^2-1)\right)\,\,,\quad P_2=\left(a^2 \sqrt{e^2-1},0,\frac{a^2}{2} (e^2-1)\right)\,\,,
$$
belong to the connected curve $\gamma$. Next, direct substitution in \eqref{eq:k1-u-ellipsoidtris} gives
$$
k_1^2(P_1)=\frac{\left(b^2 e^2-a^2 \left(e^2-1\right)\right)^2}{a^8 e^2
   \left(e^2-1\right)}\,\,,\quad k_1^2(P_2)=\frac{\left(a^2 e^2-b^2 \left(e^2-1\right)\right)^2}{b^8 e^2
   \left(e^2-1\right)}\,\,.
$$
The condition $k_1^2(P_1)=k_1^2(P_2)$ is equivalent to a second degree equation in $e^2$ which admits no
positive solutions if $a\neq b$. Therefore in this case $k_1^2$ cannot be constant along $\gamma$. Conversely,  if $a=b$, then $\lambda_n=a^n(x^2+y^2)=a^{n+4}(e^2-1)$ and $k_1^2$ is constant.

Finally, under the hypothesis $a=b$,
the biharmonicity condition, $k_1^2-K=0$, becomes
$$
\frac{1}{a^4 e^4 \left(e^2-1\right)}=0  \, \, ,
$$
which has no solution.
\end{proof}

\section{Applications to other types of implicit surfaces}

In this section we discuss two examples where it is possible to apply the scheme used to classify biharmonic curves into quadrics.

\begin{example}
Consider the implicit surface $N^2=F^{-1}(0)$ where
$$
F(x,y,z)=\frac{z^{2 n}}{c^2}+\left(x^2+y^2\right)^n-1\,\,,\quad c >0\,\,,\, \, \,\, n \geq 1 \,\, .
$$
The surface $N^2$ is a surface of revolution that, for $n=1$ and $c\neq 1$, reduces to an ellipsoid. In this case, the curves with constant Gauss curvature are the parallels given by the intersection of the surface $N^2$ with the planes $z=d=\cst$. Thus, unless $N^2$ is a sphere (i.e., $n=1=c$), the only possible biharmonic curves are
\begin{equation}\label{eq:n-quadric-intersection}
\gamma:\;\begin{cases}
F(x,y,z)= \dfrac{z^{2 n}}{c^2}+\left(x^2+y^2\right)^n-1=0\,\,,\quad c>0\,\,,\vspace{3mm}\\
G(x,y,z)=z-d=0\,\,, \quad d<\sqrt[n]{c}\,\,.
\end{cases}
\end{equation}
Using \eqref{eq:gauss-curvature} we can compute the Gauss curvature of $N^2$ along $\gamma$:
$$
K=\frac{c^4 (2 n-1) A^{2 n} d^{2 n+2} \left(c^2 A^n+d^{2 n}\right)}{\left(c^4
   d^2 A^{2 n}+A d^{4 n}\right)^2}\,\,,
$$
where
$$
A=\left(1-\frac{d^{2 n}}{c^2}\right)^{\frac{1}{n}}\,\,.
$$
Next, computing the geodesic curvature of $\gamma$ by means of \eqref{eq:geodesic}, we find
$$
k_1^2=\frac{c^4 A^{2 n} d^{4 n}}{\left(c^2-d^{2 n}\right)^2 \left(c^4 d^2 A^{2 n}+A
   d^{4 n}\right)}\,\,.
$$
Finally, the condition of biharmonicity, that is $k_1^2=K$, for a parallel \eqref{eq:n-quadric-intersection} becomes:
\begin{equation}\label{eq:n-quadric-biharmonic}
2 c^4 (1-n) d^{2 n}+d^{6 n-2} \left(1-\frac{d^{2 n}}{c^2}\right)^{\frac{1-2
   n}{n}}-c^6 (2 n-1) \left(1-\frac{d^{2 n}}{c^2}\right)=0\,\,.
\end{equation}
Although it is not easy to write down the explicit solutions of \eqref{eq:n-quadric-biharmonic} as a function $d=d(n,c)$, we point out that \eqref{eq:n-quadric-biharmonic} admits a solution $d_0\in[0,\sqrt[n]{c})$ for any $c>0$ and $n\geq 1$. To see this, we observe that the left-hand side of \eqref{eq:n-quadric-biharmonic} is continuous in $d$, assumes a negative value for $d=0$ and diverges to $+\infty$ as $d$ tends to $\sqrt[n]{c}$ .
\end{example}

\begin{example}
In this example we consider the case of graphs of revolution. Thus we assume that $N^2=F^{-1}(0)$ with
$$
F(x,y,z)=z-f(\sqrt{x^2+y^2})\,\,,
$$
for some differentiable function $f$. As in the previous example, the only curves such that the restriction of the Gauss curvature of $N^2$ is  constant are the parallels $z=d=\cst$. If we put $\rho=\sqrt{x^2+y^2}$, the Gauss curvature of $N^2$ along a parallel $z=d=f(\rho)$ and the geodesic curvature are respectively
$$
K=\frac{f'(\rho ) f''(\rho )}{\rho  \left(f'(\rho )^2+1\right)^2}\,\, ,\quad k_1^2=\frac{1}{\rho ^2 \left(f'(\rho )^2+1\right)}\,\, .
$$
It follows that a parallel $\rho=\rho_0$ is biharmonic if and only if
\begin{equation}\label{eq:graf-biharmonic}
f'(\rho_0)^2-\rho_0  f'(\rho_0 ) f''(\rho_0 )+1=0\,\,.
\end{equation}

Moreover, if $f$ is a solution of the ODE
\begin{equation}\label{eq:ode-all-parallels-biharmonic}
f'(\rho)^2-\rho  f'(\rho) f''(\rho)+1=0\,\,,
\end{equation}
then all parallels are biharmonic. The solution of \eqref{eq:ode-all-parallels-biharmonic} can be explicitly computed and is given by
\begin{equation}\label{eq:all-parallels-biharmonic}
f(\rho)=\frac{1}{2} \left(\rho \sqrt{e^{2 {c_1}} \rho^2-1}-e^{-{c_1}} \log \left(2
   e^{{c_1}} \left(\sqrt{e^{2 {c_1}} \rho^2-1}+e^{{c_1}}
   \rho\right)\right)\right)+{c_2}\,\, ,\quad c_1,c_2\in\R\,\, .
\end{equation}
We remark that the surface of revolution with the property that all its parallels  are biharmonic was already found in \cite{RCSMPP1} using different methods and afterwords J.~Monterde, in \cite{JM}, proved that it is the only surface in $\R^3$ with the property that all the level curves of the Gauss curvature are proper biharmonic
and the gradient lines of the Gauss curvature are geodesics.
\end{example}

\end{document}